\documentclass[12pt]{amsart}
\usepackage{amssymb, amsmath, amsthm, amsfonts, amscd}
\usepackage{xcolor}
\usepackage{datetime}
\usepackage{mathabx}
\usepackage{verbatim}
\usepackage{enumerate}
\newtheorem{theorem}{Theorem}[section]
\newtheorem{lemma}[theorem]{Lemma}
\newtheorem{proposition}[theorem]{Proposition}
\newtheorem{corollary}[theorem]{Corollary}
\theoremstyle{definition}
\newtheorem{definition}[theorem]{Definition}
\newtheorem{example}[theorem]{Example}

\theoremstyle{remark}
\newtheorem{remark}[theorem]{Remark}
\numberwithin{equation}{section}

\title[ Normaloid essential isometric operators]{Hyperinvariant subspaces for normaloid essential isometric operators}
\author{Neeru Bala}
\address{Department of Mathematics and Computing, Indian Institute of Technology (ISM) Dhanbad, Dhanbad, Jharkhand, India 826 004}
\email{neerusingh41@gmail.com, neerubala@iitism.ac.in}

\author{Ramesh Golla}
\address{Department of Mathematics, Indian Institute of Technology - Hyderabad, Kandi, Sangareddy, Telangana, India 502 284.}
\email{rameshg@math.iith.ac.in}

\subjclass[2010]{47A10, 47A15, 47A53; 47B07}
\keywords{Absolutely norm (minimum) attaining operator, essentially normal, invariant subspace, hyperinvariant subspace.}
\date{\currenttime ;  \today}
\begin{document}
\begin{abstract}

In this article, we prove the existence of a non-trivial hyperinvariant subspace for a subclass of compact perturbations of scalar multiple of a partial isometry.  Later, we illustrate that
this class contains several important classes of operators. As a consequence, we prove that a Schatten class perturbation of a partial isometry with finite-dimensional null space has a non-trivial hyperinvariant subspace. 
\end{abstract}

\maketitle

	\section{Introduction}
	Throughout the article, we assume $H,H_1,H_2$ to be separable infinite dimensional complex Hilbert spaces and $\mathcal{B}(H_1,H_2)$ denote the space of all bounded linear operators from $H_1$ to $H_2$, in particular, $B(H,H):=B(H)$. A closed subspace $M\subseteq H$ is called an \textit{invariant} subspace for $T\in\mathcal{B}(H)$, if $T(M)\subseteq M$. If $M$ is invariant under all the operators in $\mathcal{B}(H)$ which commutes with $T$, then $M$ is called a \textit{hyperinvariant} subspace for $T$.
	
	One of the famous problems in operator theory is the invariant subspace problem, which can be stated as:\\
	"Does every non-zero bounded linear operator which is not a scalar multiple of the identity operator on a Hilbert space have a non-trivial invariant (or hyperinvariant) subspace?"
	
	%Similarly, the hyperinvariant subspace problem can be stated as \\``Does every non zero bounded operator which is not a scalar multiple of the identity has a non-trivial hyperinvariant subspace?"
	
	%	For Banach space operators, the answer is negative and a counter-example was given by %Enflo \cite{Enflo}.
	
	The invariant (or hyperinvariant) subspace problem has been solved for many subclasses of $\mathcal{B}(H)$, for example, compact operators \cite{SMI} and isometries \cite{DOU}. Some of the famous techniques used to find a non-trivial invariant subspace are fixed point theory or Lomonosov's technique \cite{LOM}, Brown's technique of functional calculus \cite{BRO} and non-standard analysis which is used by Bernstein and Robinson \cite{BER}.  For more details about the invariant and hyperinvariant subspaces, we refer to \cite{Chalendar,RADJAVI}.
	
	%Recently, Ambrozie and Muller \cite{MUL} showed that polynomially bounded operators on Hilbert space, whose spectrum contains the unit circle have invariant subspace. C. Foias \textit{et.al} \cite{FOIA} proved that some special class of subnormal operators have a non-trivial hyperinvariant subspace, while Liu \cite{LIU} gave the converse of Lomonosov's theorem and obtained some necessary and sufficient conditions for invariant subspace problem.
	
	In this article, our main aim is to study the hyperinvaraint subspace problem for a subclass of normaloid essentially isometric operators (or a scalar multiple of normaloid essentially isometric operators). An operator $T\in\mathcal{B}(H)$ is called essentially isometric if $I-T^*T$ is a compact operator. Explicitly saying, we prove that a normaloid essentially isometric operator $T$ on a Hilbert space has a non-trivial hyperinvariant subspace if eigenvalues of $T^*T$ smaller than $1$ form a p-summable series for some $1\leq p<\infty$. This improves Corollaries $6.16$ and $6.17$ of \cite{RADJAVI}.
	
	The class of absolutely norm attaining operators ($\mathcal{AN}$-operators), which was introduced by Carvajal and Neves in \cite{CAR1} lies in the class of scalar multiple of essentially isometric operators. 	Another class which is also a scalar multiple of essentially isometric operator is the class of absolutely minimum attaining operators (or $\mathcal{AM}$-operators), which is defined in a similar way as the class of $\mathcal{AN}$-operators but they differs significantly. For example, every $\mathcal{AM}$-operator has a closed range, while it is not the case with the class of $\mathcal{AN}$-operators. Another notable difference is the spectral behavior of these two classes, see \cite{CAR1,GAN,GAN1} for more details. 	Existence of hyperinvariant subspaces for Toeplitz and Hankel in the class of $\mathcal{AN}$ and $\mathcal{AM}$-operators  are studied in \cite{GRSSS1,GRSSS2}. % while hyperinvariant subspaces for Hankel $\mathcal{AN}$-operators and Toeplitz
	%$\mathcal{AM}$-operators are studied in \cite{GRSSS2}. 
	In this article,  we show the existence of a non-trivial hyperinvariant subspace for normaloid  $\mathcal{AN}$-operators as well as minimaloid $\mathcal{AM}$-operators, which improve the results of \cite{RAMpara}, where the author have studied existence of a reducing subspace for paranormal $\mathcal{AN}$-operators.
	For more details about $\mathcal{AN},\, \mathcal{AM}$-operators, we refer to \cite{PAL,RAMpara,VENKU} and references therein.

	This article is divided into four sections. Section $2$, consists of the basic terminology and results which are used throughout the article. In sections $3$,  we show the existence of a non-trivial hyperinvariant subspace for a subclass of normaloid essentially isometric operators. In section $4$, we prove the existence of non-trivial hyperinvariant subspace for normaloid  $\mathcal{AN}$-operators and minimaloid $\mathcal{AM}$-operators, and extend the same for the operator norm closure of these class of operators.
 Finally, we prove that a Schatten class perturbation of a partial isometry with finite-dimensional null space has a non-trivial hyperinvariant subspace.
\section{Notations and preliminaries}

 The  null space and range space of $T\in \mathcal B(H)$ are denote by $N(T)$ and $R(T)$, respectively. %If $R(T)$ is finite-dimensional, then $T$ is called a \textit{finite-rank} operator. If $T$ maps every bounded set in $H$ onto a pre-compact set in $H$, then $T$ is called a \textit{compact} operator. We denote the set of all finite-rank and compact operators in $B(H)$ by $\mathcal{F}(H)$ and $\mathcal{K}(H)$, respectively. 
In this article, we frequently use properties of the Calkin algebra $\mathcal{B}(H)/\mathcal{K}(H)$, where $\mathcal{K}(H)$ denotes the space of all compact operators in $\mathcal{B}(H)$. Consider the natural homomorphism
$$\pi:\mathcal{B}(H)\rightarrow \mathcal{B}(H)/\mathcal{K}(H).$$
Then the  \textit{essential spectrum} of $T$ is defined by  $\sigma_{\text{ess}}(T)=\sigma(\pi(T))$, where $\sigma(S)$ is the  spectrum of an operator $S\in\mathcal{B}(H)$. 

There is an another way of defining the essential spectrum of an operator, that is using Fredholm theory. An operator $T\in\mathcal{B}(H)$ is called \textit{Fredholm}, if $R(T)$ is closed and $N(T)$, $N(T^*)$ are finite dimensional. In this case, the index of $T$ is defined by
$$\text{ind}(T)=\dim(N(T))-\dim(N(T^*)).$$
Note that,
\begin{align*}
	%	\sigma_{w}(T)&=\{\lambda\in\mathbb{C}:T-\lambda I\text{ is not Fredholm of index }0\},\\
	\sigma_{\text{\text{ess}}}(T)&=\{\lambda\in\mathbb{C}:T-\lambda I\text{ is not Fredholm}\}.
\end{align*}

The famous Weyl's theorem states that if $A,\,B\in\mathcal{B}(H)$ with $A-B\in \mathcal K(H)$, then $\sigma_{\text{ess}}(A)=\sigma_{\text{ess}}(B)$ (see \cite[Corollary 5.6]{Berbarian}).
%We say $T\in\mathcal{B}(H)$, satisfies Weyl's theorem if $$\sigma(T)\setminus\sigma_w(T)=\pi_{00}(T).$$
For a self-adjoint operator $T\in\mathcal{B}(H)$, the set $\sigma_d(T)=\sigma(T)\setminus\sigma_{\text{ess}}(T)$ is called the \textit{discrete spectrum} of $T$.
For more details about Fredholm theory, we refer to \cite{MULL}.

In this article, we also use a special subclass of compact operators, that is the Schatten $p$-class.
\begin{definition}
   The singular values $s_n(A) \; (n\in \mathbb{N})$ of  $A\in \mathcal K(H)$ are the eigenvalues of $|A|=(A^{*}A)^{\frac{1}{2}}$. We say $A$  is in the Schatten $p$-class $\mathcal K_{p}(H),\; (1\leq p<\infty)$, if 
 \begin{equation*} 
 %\|A\|_{p}=
  \sum_{n=1}^{\infty}s_n(A)^{p}<\infty.
 \end{equation*}
 \end{definition}

It is easy to observe that the class of finite rank operators, $\mathcal{F}(H)\subseteq\mathcal{K}_p(H)$ for $1\leq p<\infty$.

%For $T\in\mathcal{B}(H)$, $\rho(T)=\{\lambda\in\mathbb{C}:\lambda I-T\text{ is invertible in }\mathcal{B}(H)\}$ is called the \textit{resolvent set} of $T$ and $\sigma(T)=\mathbb{C}\setminus\rho(T)$ is the \textit{spectrum} of $T$. The \textit{point spectrum } of $T$ is the set
%$$ \sigma_p(T)=\{\lambda\in\sigma(T):\,T-\lambda I\text{ is not injective}\}.$$

%An operator $T\in\mathcal{B}(H)$ is called \textit{Fredholm}, if $R(T)$ is closed and $N(T)$, $N(T^*)$ are finite dimensional. In this case, the index of $T$ is defined by
%$$\text{ind}(T)=\dim(N(T))-\dim(N(T^*)).$$
%The  \textit{essential spectrum} of $T$ are defined by
%\begin{align*}
%	\sigma_{w}(T)&=\{\lambda\in\mathbb{C}:T-\lambda I\text{ is not Fredholm of index }0\},\\
%	\sigma_{\text{\text{ess}}}(T)&=\{\lambda\in\mathbb{C}:T-\lambda I\text{ is not Fredholm}\},
%\end{align*}
%respectively. Note that, $\sigma_{\text{ess}}(T)=\sigma(\pi(T))$. Weyl's theorem states that if $A$ and $B$ are normal operators with $A-B\in \mathcal K(H)$, then $\sigma_{\text{ess}}(A)=\sigma_{\text{ess}}(B)$ (see \cite[Corollary 5.6]{Berbarian}).
%We say $T\in\mathcal{B}(H)$, satisfies Weyl's theorem if $$\sigma(T)\setminus\sigma_w(T)=\pi_{00}(T).$$
%For a self-adjoint operator $T\in\mathcal{B}(H)$, the set $\sigma_d(T)=\sigma(T)\setminus\sigma_{\text{ess}}(T)$ is called the \textit{discrete spectrum} of $T$.
%For more details about Fredholm theory, we refer to\cite{MULL}.

Throughout the article, we deal with normaloid essentially isometric operators, where by {\it normaloid}, we mean operators $T\in\mathcal{B}(H)$ for which the spectral radius $r(T)=\|T\|$, where  $r(T):=\sup\{|\lambda|:\lambda\in\sigma(T)\}$, equivalently $$r(T)=\underset{n\rightarrow\infty}{\limsup}\|T^n\|^{1/n}.$$

An operator $T\in\mathcal{B}(H)$ is said to be \textit{essentially isometric} if $\pi(T)$ is an isometry. In other words, $T$ is essentially isometric if $I-T^*T$ is a compact operator.

%It is known that $\sigma_{\text{ess}}(T)\subseteq\sigma_w(T)$ and since $\sigma_{\text{ess}}(T)\ne\emptyset$, we have that $\sigma_{\omega}(T)\ne\emptyset$. Equivalently, the Weyl spectrum can also be defined as
%$$\sigma_{\omega}(T)=\underset{K\in\mathcal{K}(H)}{\cap}\sigma(T+K).$$

Here we mention a few more notations, which are frequently used in the subsequent sections. If $M$ is a subspace of $H$, then we denote the %restriction of $T$ to $M$ by $T|_M$ and the 
unit sphere of $M$ by $S_M=\{x\in M:\|x\|=1\}$. For any $r\in\mathbb{R}^+\cup\{0\}$, we denote the circle and ball of radius $r$ in $\mathbb{C}$ by $C(0,r):=\{z\in\mathbb{C}:|z|=r\}$ and $B(0,r):=\{z\in\mathbb{C}:|z|\leq r\}$, respectively.

\section{Normaloid essentially isometric opertors}	
In this section, we study a connection between the spectrum  and existence of non-trivial hyperinvariant subspace for an essentially isometric operator.

We start this section with a small observation about the essential spectrum of an operator.
	\begin{lemma}\label{lemma essspe}
		Let $T\in\mathcal{B}(H)$ and $N(T)=N(T^*)$. Then $\sigma_{\text{ess}}(T^*T)=\sigma_{\text{ess}}(TT^*)$.
	\end{lemma}
	\begin{proof}
		By \cite[Theorem 6, Page 173]{MULL}, we know that $$\sigma_{\text{ess}}(T^*T)\setminus\{0\}=\sigma_{\text{ess}}(TT^*)\setminus\{0\}.$$ It is enough to show that, if $0\in\sigma_{\text{ess}}(T^*T)$ then $0\in\sigma_{\text{ess}}(TT^*)$. If $0\in\sigma_{\text{ess}}(T^*T)$, then by definition either $R(T^*T)$ is not closed or $N(T)$ is infinite dimensional. By the assumption, this is equivalent to say that either $R(TT^*)$ is not closed or $N(T^*)$ is infinite-dimensional, respectively. As a result, $0\in\sigma_{\text{ess}}(TT^*)$.
	\end{proof}
Now, we discuss the main result of this section regarding existence of hyperinvariant subspace.	This is a generalization of \cite[Corollary 6.16]{RADJAVI}.
\begin{theorem}\label{Main thm2}
	Let $T\in\mathcal{B}(H)$ be an operator which is not a scalar multiple of the identity and satisfies the following conditions.
	\begin{enumerate}
		\item  $\sigma_{\text{ess}}(T^*T)=\{\alpha\}$ for a non-negative real number $\alpha$,
		\item \label{condition3} ${\{ (\alpha-\lambda_n): \lambda_n \in \sigma_d(T^*T) }\}\in \ell^p(\mathbb{N})$ for some $1\leq p<\infty.$
	\end{enumerate}  Then $T$ has a non-trivial hyperinvariant subspace.
\end{theorem}
\begin{proof} 

If $\alpha=0$, then $T$ is a compact operator and consequently has a non-trivial hyperinvariant subspace. Thus, we assume that $\alpha\ne 0$.
	We also assume that  $N(T)=\{0\}=N(T^*)$, otherwise $N(T)$ or $\overline{R(T)}$ is a non-trivial hyperinvariant subspace for $T$. By Lemma \ref{lemma essspe}, we have $\sigma_{\text{ess}}(T^*T)=\{\alpha\}=\sigma_{\text{ess}}(TT^*)$. Let $T=U|T|$ be the polar decomposition of $T$, where $U$ is a unitary operator.  %Since $\sqrt{\alpha}^{-1}\pi(T)$ is a unitary operator, we have $\sigma_{\text{ess}}(T)\subseteq C(0,\sqrt{\alpha})$. Since $\sigma_{\text{ess}}(T^*T)={\{\alpha}\}$, 
	By \cite[Lemma 3]{Kover1}, we get that $\sigma_{\text{ess}}(|T|)={\{\sqrt{\alpha}}\}$ and %Since $\sigma_{\text{ess}}(\sqrt{\alpha}I)={\{\sqrt{\alpha}}\}$, by 
	\cite[Theorem 1]{Sikonoa} implies that there exists a unitary operator $V$ and $K\in \mathcal K(H)$ such that $|T|=K+V^*(\sqrt{\alpha }I)V=K+\sqrt{\alpha}I$. Then
 \begin{equation}\label{cptpertunitary}
 T=U|T|=U(K+\sqrt{\alpha}I)=K_1+\sqrt{\alpha}U,
 \end{equation}
 where $K_1=UK$, a compact operator. By the Weyl's theorem, we conclude that $\sigma_{\text{ess}}(T)=\sigma(\sqrt{\alpha}U)\subseteq C(0,\sqrt{\alpha})$. The given condition $(2)$ implies that $K=|T|-\sqrt{\alpha}I \in \mathcal K_{p}(H)$ for $1\leq p<\infty$ and as a result $K_1\in \mathcal K_{p}(H)$.

 %Also, note that the assumption on the discrete spectrum of $T^*T$ yield that $K=|T|-\sqrt{\alpha}I \in \mathcal K_{p}(H)$ for $1\leq p<\infty$. From this, we can also conclude that $K_1\in \mathcal K_{p}(H)$.  
	
%	Here we consider the following cases.
	
	%\noindent Case {(i)}: $\sigma(T)=\sigma_{\text{ess}}(T)$:\\
	We assume that  $\sigma(T)=\sigma_{\text{ess}}(T)$ and $\sigma(T)$ is connected. Because, if $\sigma(T)$ is not connected, then the Riesz-Dunford functional calculus gives existence of a non-trivial hyperinvariant subspace for $T$, and if $\sigma(T)\ne\sigma_{\text{ess}}(T)$, then for $\lambda \in \sigma(T)\setminus \sigma_{\text{ess}}(T)$, we have $T-\lambda I$ is a Fredholm operator, which gives either $N(T-\lambda I)$ or $R(T-\lambda I)$ is an hyperinvariant subspace for $T$. %Here we assume that $\sigma(T)$ is connected, otherwise by the Riesz-Dunford functional calculus we know the existence of a hyperinvariant subspace for $T$.
%	In this case, $\sigma(T)\subseteq C(0,\sqrt{\alpha})$. If $\sigma(T)$ is disconnected, then by the Riesz-Dunford functional calculus we know the existence of a hyperinvariant subspace for $T$. 
 %Next, assume that  $\sigma(T)$ is connected. 
  
    If $\sigma(T)$ is a singleton set, say $\{\beta\}$, then  $\sigma_{\text{ess}}(U)={\{\frac{\beta}{\sqrt{\alpha}}}\}=\sigma_{\text{ess}}(\frac{\beta}{\sqrt{\sqrt{\alpha}}}I)$. By \cite[Theorem 1]{Sikonoa}, there exists a unitary operator $W$ and a compact operator $K_2$ such that 
 \begin{equation}\label{unitaryeqn}
 U=K_2+W^*\frac{\beta}{\sqrt{\alpha}} W=K_2+\frac{\beta}{\sqrt{\alpha}}I.
 \end{equation}
 By equations (\ref{cptpertunitary}) and (\ref{unitaryeqn}), we get that
	\begin{equation*}
		T=\sqrt{\alpha} U+K_1=\sqrt{\alpha} (K_2+\frac{\beta}{\sqrt{\alpha}} I)+K_1=\beta I+\tilde{K},
	\end{equation*}
	where $\tilde{K}=\sqrt{\alpha} K_2+K_1\in \mathcal K(H)$. By the Lomonosov's theorem \cite[Corollary 8.25]{RADJAVI}, we know that $\tilde{K}$ has a non-trivial hyperinvariant subspace and so is $T$. 
	
		Next, we assume that $\sigma(T)$ contains more than one point. In this case, the operator $T$ given by Equation (\ref{cptpertunitary}) satisfies the conditions of \cite[Corollary 6.16, Page 107]{RADJAVI}. Hence  $T$ has a non-trivial hyperinvariant subspace.
  %On the other hand, assume that $\sigma_{\text{ess}}(T)\subset\sigma(T)$. For $\lambda \in \sigma(T)\setminus \sigma_{\text{ess}}(T)$, we know that $T-\lambda I$ is a Fredholm operator.
	%\noindent Case {(ii)}: $\sigma_{\text{ess}}(T)\subset \sigma(T)$:\\ Let $\lambda \in \sigma(T)\setminus \sigma_{\text{ess}}(T)$. Then $T-\lambda I$ is Fredholm. That is, $R(T-\lambda I)$ is closed, and both $N(T-\lambda I)$ and $N(T^*-\bar{\lambda }I)$ are finite dimensional. If both the spaces $N(T^*-\bar{\lambda}I)$ and  $N(T-\lambda I)$ are zero sets, then $\overline{R(T-\lambda I)}=N(T^*-\bar{\lambda I})^{\bot}=H$. This implies  $T-\lambda I$ is an invertible operator, which is a contradiction to the fact that $\lambda \in \sigma(T)$. As a result, either $N(T-\lambda I)$ or $N(T^*-\bar{\lambda} I)$ is a non-zero set. If $N(T-\lambda I)\neq {\{0}\}$, then it is a hyperinvariant subspace for $T$. If $N(T^*-\bar{\lambda}I)\neq {\{0}\}$, then $R(T-\lambda I)$ is a non trivial  hyperinvariant subspace for $T$. This completes the proof.
\end{proof}

By looking at the proof of Theorem \ref{Main thm2}, it is easy to obtain the following.
\begin{corollary}
	Let $T\in\mathcal{B}(H)$ be such that $\sigma(T)=\sigma_{\text{ess}}(T)$ is connected  and $\sigma_{\text{ess}}(T^*T)$ is a singleton set. Then either of the following conditions  gives the existence of a non-trivial hyperinvariant subspace for $T$.
	\begin{enumerate}
		\item $\sigma(T)$ is a singleton set.
		\item ${\{ (\alpha-\lambda_n): \lambda_n \in \sigma_d(T^*T) }\}\in \ell^p(\mathbb{N})$ for some $1\leq p<\infty.$
	\end{enumerate}
\end{corollary}

%\begin{corollary}
 %   Let $T\in\mathcal{B}(H)$ be such that 
  %  \begin{enumerate}
   %  \item $\sigma_{\text{ess}}(T^*T)$ is a single point.
    %\item $\sigma(T)=\sigma_{\text{ess}}(T)$ is a single point.
    %\item $\sigma(T)$ is connected.
    %     \end{enumerate}
     %    Then $T$ has a non-trivial hyperinvariant subspace.
    %\end{corollary}
     % \begin{corollary}
     % Let $T\in\mathcal{B}(H)$ be such that 
    %\begin{enumerate}
   
    %\item $\sigma_{\text{ess}}(T^*T)$ is a single point.
    %\item $\sigma(T)=\sigma_{\text{ess}}(T)$ contains atleast two points.
    %\item $\sigma(T)$ is connected.
    %\item ${\{ (\alpha-\lambda_n): \lambda_n \in \sigma_d(T^*T) }\}\in \ell^p(\mathbb{N})$ for some $1\leq p<\infty.$
    %Then $T$ has a non-trivial hyperinvariant subspace.
     %\end{enumerate}
  %\end{corollary}

\begin{corollary}
    Let $S\in \mathcal B(H)$ be an isometry and $K\in \mathcal K_{p}(H)$ for some $1\leq p<\infty$. Then $S+K$ has a non-trivial hyperinvariant subspace.
\end{corollary}
\begin{proof}
Define  $T=S+K$. Then we have $T^*T=I+\tilde{K}$, where $\tilde{K}=S^*K+K^*S+K^*K$. It is easy to see that $\tilde{K}\in \mathcal K_{p}(H)$, $\sigma_{\text{ess}}(T^*T)={\{1}\}$, and  $\sigma_{d}(T^*T)={\{1+\lambda_n: \lambda_n \in \sigma_{d}(\tilde{K})}\}.$ %. Also, by the Weyl's theorem, we have $\sigma_{\text{ess}}(T^*T)={\{1}\}$. Since $\sigma(T^*T)=\sigma_{d}(T^*T)\cup \sigma_{\text{ess}}(T^*T)$, we get that $\sigma_{d}(T^*T)={\{1+\lambda_n: \lambda_n \in \sigma_{d}(\tilde{K})}\}$. Since $\tilde{K}\in \mathcal K_{p}(H)$, it is clear that $(\lambda_n)\in \ell^{p}(\mathbb{N})$.
  Now, the result follows from Theorem \ref{Main thm2}.
\end{proof}
We can also obtain \cite[Corollary 6.17]{RADJAVI} from Theorem \ref{Main thm2}.

\begin{corollary}
	Let $T\in \mathcal B(H)$ be such that $1-T^*T \in \mathcal K_{p}(H)$ for some $1\leq p<\infty$. Then $T$ has non-trivial hyperinvariant subspace.
\end{corollary}
\begin{proof}
	Let $I-T^*T=K\in \mathcal K_{p}(H)$. Then we have $\sigma_{\text{ess}}(T^*T)={\{1}\}$. Since $K\in \mathcal K_p(H)$, let ${\{\mu_n:n\in \mathbb N}\}$ be the spectrum of $K$. Then each $\mu_n$ is an eigenvalue with finite multiplicity and $\sigma(T^*T)={\{1-\mu_n:n\in \mathbb N}\}$. Hence by Theorem \ref{Main thm2}, $T$ has a non-trivial hyperinvariant subspace.
\end{proof}

\begin{theorem}\label{Main thm}
	Let $T\in\mathcal{B}(H)$ satisfies the following conditions:
	\begin{enumerate}
		\item $T$ is normaloid,
		\item  $\sigma_{\text{ess}}(T^*T)=\{\alpha\}$ for a non-negative real number $\alpha$,
		\item\label{condition3} $\{ \alpha -\lambda_i:\lambda_i\in\sigma_d(T^*T),\,\lambda_i<\alpha \}\in\ell^p(\mathbb{N})$ for some $1\leq p<\infty.$ %where $\sigma_p(T^*T)$ denotes the point spectrum of $T^*T$.
	\end{enumerate}  Then $T$ has a hyperinvariant subspace.
\end{theorem}
\begin{proof} By a similar argument as in the proof of Theorem \ref{Main thm2}, we assume that $\alpha>0$ and $N(T)=\{0\}=N(T^*)$.% By Lemma \ref{lemma essspe}, we have $\sigma_{\text{ess}}(T^*T)=\{\alpha\}=\sigma_{\text{ess}}(TT^*)$ and $\alpha\ne 0$. Since $\sqrt{\alpha}^{-1}\pi(T)$ is a unitary operator, we have $\sigma_{ess}(T)\subseteq C(0,\sqrt{\alpha})$.
	
	First we assume that $\sqrt{\alpha}<\|T\|$.  As $T$ is normaloid, there exists  $\lambda\in\sigma(T)$ such that $|\lambda|=\|T\|$.  Thus $T-\lambda I$ is a Fredholm operator and consequently $R(T-\lambda I)$ is closed. But $\lambda\in\sigma(T)$, we have either $N(T-\lambda I)\ne\{0\}$ or $R(T-\lambda I)\ne H$ %. As a result either $N(T-\lambda I)$ or $R(T-\lambda I)$ 
	is a non-trivial hyperinvariant subspace for $T$.
	%Also $\|T\|\notin\sigma_{\omega}(T)$, otherwise $\left(B(0,\|T\|)\setminus \overline{B(0,\alpha)} \right)\cap\sigma(T)$ becomes uncountable, as index is locally constant. Thus $T-\|T\|I$ is a Fredholm operator of index zero, consequently $\|T\|$ is an isolated eigenvalue of finite multiplicity.
	%So $N(T-\|T\|I)$ as well as $R(T-\|T\|I)$ are hyperinvariant subspaces for $T$.
	
	If $\sqrt{\alpha}=\|T\|$, then $T$ satisfies the conditions of Theorem \ref{Main thm2}. Consequently, $T$ has a non-trivial hyperinvariant subspace.
 \end{proof}

	The following result is a special case of the above result when $\alpha=1$.
	\begin{corollary}
		Let $T\in\mathcal{B}(H)$ be a normaloid essentially isometric operator. If the set
		$$\{1-\lambda:\lambda\in\sigma_d(T^*T),\,\lambda<1\} $$ is $p$-summable for some $1\leq p<\infty$, then $T$ has a non-trivial hyperinvariant subspace.
	\end{corollary}
	\begin{proof}
	If $T$ is an essentially isometry, then $\sigma_{\text{ess}}(T^*T)=\{1\}$. The result follows from Theorem \ref{Main thm}.
	\end{proof}
  The following example illustrates the fact that the normaloid condition in Theorem \ref{Main thm} is a sufficient condition but not necessary.
  \begin{example}\label{example AN}
  	Consider the operator $T:\mathbb{C}^2\oplus\ell^2(\mathbb{N})\rightarrow \mathbb{C}^2\oplus\ell^2(\mathbb{N})$, defined by
  	$$T=A\oplus R,$$
  	where $R$ is the right shift operator on $\ell^2(\mathbb{N})$ and 
  	\[A=\begin{bmatrix}
  		0&2\\
  		0&0
  	\end{bmatrix}.
  	\]
  	It is easy to see that $\|T\|=2$ and $\sigma(T)=\{z\in\mathbb{C}:|z|\leq 1\}$. Hence $r(T)=1<\|T\|$ and consequently $T$ is not normaloid. Note that $\sigma_{\text{ess}}(T^*T)=\{1\}$. If \[S=\begin{bmatrix}
  		X&Y\\
  		Z&W
  	\end{bmatrix}\in\mathcal{B}(\mathbb{C}^2\oplus\ell^2(\mathbb{N}))\]
  	is an operator commuting with $T$, then
  	\begin{align*}
  		AX=XA,\,AY=YR,\,RZ=ZA,\text{ and }RW=WR.
  	\end{align*}
  	The condition $RZ=ZA$ implies that $ZA=R^*ZA^2=0$. Consequently, either $ZA=RZ=0$ or $Z=0$. Hence $\mathbb{C}^2\oplus\{0\}$ is a non-trivial hyperinvariant subspace for $T$. This is an example of an essentially isometric non-normaloid operator, which has a non-trivial hyperinvariant subspace. Observe that $\sigma_{d}(T^*T)=\sigma_{d}(A^*A)={\{0,4}\}$ and the point spectrum $\sigma_p(T^*T)={\{0,1,4}\}$.
  \end{example}

  \section{Absolutely norm attaining operators}
 In this section, we give examples of operators which comes under the class of operators discussed in section $3$, namely absolutely norm attaining operators and absolutely minimum attaining operators. %To be precise, we discuss the existence of hyperinvariant subspaces of absolutely norm attaining operators, which form a nice subclass of operators discussed in the previous section.
  
  Recall that  $T\in \mathcal{B}(H_1,H_2)$ is said to be \textit{norm attaining}, if there exist $x\in S_{H_1}$, such that $\|T\|=\|Tx\|$. We say $T$ is an \textit{absolutely norm attaining}, if for every non-zero closed subspace $M\subseteq H_1$, $T|_M:M\rightarrow H_2$ is a norm attaining operator. We denote the class of all absolutely norm attaining operators by $\mathcal{AN}(H_1,H_2)$. %If $H_1=H_2=H$, then $\mathcal{AN}(H):=\mathcal{AN}(H,H).$ 
     For more details about $\mathcal{AN}$-operators, we refer to \cite{CAR1,PAL,RAMpara,VENKU}. From \cite[Theorem 5.1]{PAL}, we have the following characterization for positive $\mathcal{AN}$-operators.

  \begin{theorem}\cite[Theorem 5.1]{PAL}\label{ANrepresentation}
  	Let $T\in\mathcal{B}(H)$ be a positive operator. Then
$T\in \mathcal{AN}(H)$  if and only if  $T = \alpha I +K+F$, where $\alpha \geq 0$, $K$ is a positive compact operator and $F$ is self-adjoint finite-rank operator.
   \end{theorem}

  %\begin{theorem}
  % Then the following are equivalent.
 % 	\begin{enumerate}
 % 		\item $T\in\mathcal{AN}(H)$ and positive.
  %		\item There exist unique $K\in\mathcal{K}(H)_+,\,F\in\mathcal{F}(H)_+$ and a real number %$\alpha\geq 0$ satisfying the following;
 % 		\begin{enumerate}
  %			\item $0\leq F\leq\alpha I$,
  %			\item $KF=0$,
 % 		\end{enumerate}
 % 		$T=\alpha I+K-F$.
%  	\end{enumerate}
%  \end{theorem}
  
 In the following result, we try to connect the $\mathcal{AN}$-property of an operator with its adjoint and this improves Theorem 2.7 of \cite{RAMpara}.
   
  %Using the essential spectrum, we try to connect the $\mathcal{AN}$-property of an operator with its adjoint. Our next result improves Theorem 2.7 of \cite{RAMpara}.
  \begin{proposition}\label{prop essan}
  	Let $T\in\mathcal{B}(H_1,H_2)$. Then any two conditions of the following imply the third condition.
  	\begin{enumerate}
  		\item\label{T} $T\in\mathcal{AN}(H_1,H_2)$.
  		\item\label{Tadjo}  $T^*\in\mathcal{AN}(H_2,H_1).$
  		\item\label{essspe} $\sigma_{\text{ess}}(T^*T)=\sigma_{\text{ess}}(TT^*)$.
  	\end{enumerate}
  \end{proposition}
  \begin{proof}
  	By \cite[Theorem 2.7]{RAMpara}, it is enough to show that conditions \textit{(\ref{T})} and \textit{(\ref{Tadjo})} implies condition \textit{(\ref{essspe})}. From \cite[Theorem 6, Page 173]{MULL}, we have $\sigma_{\text{ess}}(T^*T)\setminus\{0\}=\sigma_{\text{ess}}(TT^*)\setminus\{0\}$. Thus, it is enough to discuss the case of $0$.
  	
  	Let $0\in\sigma_{\text{ess}}(T^*T)$. Since $T\in\mathcal{AN}(H_1,H_2)$, we have $T^*T\in \mathcal{AN}(H_1)$ by \cite[Corollary 2.11]{VENKU}. By Theorem \ref{ANrepresentation}, $T^*T=K-F+\alpha I$ for some positive operators $K\in \mathcal K(H),\; F\in \mathcal F(H)$ and $\alpha \geq 0$, and consequently $\sigma_{\text{ess}}(T^*T)={\{\alpha}\}$. As $0\in \sigma_{\text{ess}}(T^*T)$, we have $\alpha=0$. That is, $T^*T$ is compact and consequently $T$ and $T^*$ are compact operators. Hence  $\sigma_{\text{ess}}(TT^*)={\{0}\}$. In a similar manner, the other way implication can be proved. 
  \end{proof}
   
   An operator $T\in\mathcal{B}(H)$ is called an essentially normal operator if $T^*T-TT^*$ is a compact operator. In other words, $\pi(T)^*\pi(T)=\pi(T)\pi(T)^*$ and the \textit{essential norm} of $T$ is defined by
   \begin{equation*}
   \|T\|_{e}:=\text{dist}(T,\mathcal{K}(H))=\inf{\{\|T-K\|:K\in \mathcal{K}(H)}\}.
   \end{equation*}

   The following result proves that a subclass of absolutely norm attaining operators lies in the class of essentially normal operators.
  \begin{proposition}\label{prop essnormal}
  	Let $T,T^*\in\mathcal{AN}(H)$. Then $T$ is essentially normal and $\sigma_{\text{ess}}(T)\subseteq C(0,\alpha)$, for some $\alpha\in\mathbb{R}^+\cup\{0\}$.
  \end{proposition}
  \begin{proof}
  	Suppose $T,T^*\in\mathcal{AN}(H)$. Then Proposition \ref{prop essan} and \cite[Theorem 2.4]{RAMpara} implies that $\sigma_{\text{ess}}(T^*T)=\sigma_{\text{ess}}(TT^*)=\{\alpha^2\}$, for some $\alpha\in\mathbb{R}^+\cup\{0\}$, which is equivalent to say that $\sigma(\pi(T)^*\pi(T))=\{\alpha^2\}=\sigma(\pi(T)\pi(T)^*)$. By Theorem \ref{ANrepresentation}, we have $T^*T=\alpha^2 I+K_1-F_1$ and $TT^*=\alpha^2 I+K_2-F_2$, for some positive operators $K_1,K_2\in\mathcal{K}(H)$ and $F_1,F_2\in\mathcal{F}(H)$. Thus $T^*T-TT^*\in\mathcal{K}(H)$. Consequently $T$ is essentially normal and $\pi(T)^*\pi(T)=\alpha^2 I=\pi(T)\pi(T)^*$.
  	
  	If $\alpha=0$, then $\pi(T)=0$. Thus $T,T^*\in\mathcal{K}(H)$, and $\sigma_{\text{ess}}(T)=\{0\}$.	On the other hand, if $\alpha>0$, then $$\pi\left(T/\alpha\right)^*\pi\left(T/\alpha\right)=I=\pi\left(T/\alpha\right)\pi\left(T/\alpha\right)^*.$$ Thus $\pi\left(T/\alpha\right)$ is unitary, and $\sigma\left(\pi\left(T/\alpha\right)\right)\subseteq C(0,1)$. This is equivalent to say $\sigma_{\text{ess}}(T)\subseteq C(0,\alpha)$.
  \end{proof}
  The converse of the above Theorem is not true. That is, if $T$ is essentially normal and $\sigma_{\text{ess}}(T)\subseteq C(0,\alpha)$ for some $\alpha\in\mathbb{R}^+\cup\{0\}$, then $T$ need not be an $\mathcal{AN}$-operator, see the following example.
  \begin{example}
  	Let $U$ be a unitary operator on $\ell^2(\mathbb{N})$ and $S:\ell^2(\mathbb{N})\rightarrow \ell^2(\mathbb{N})$ be the linear operator defined by%on the orthonormal basis as follows;
  	$$Se_n=\left(1-\frac{1}{n}\right)e_n,~\forall n\in\mathbb{N},$$
  	where $\{e_n:n\in\mathbb{N}\}$ is the standard orthonormal basis for $\ell^2(\mathbb{N})$. Clearly $S$ is not norm attaining and hence $S\notin\mathcal{AN}(\ell^2(\mathbb{N}))$.  Consider the operator $T:\ell^2(\mathbb{N})\oplus \ell^2(\mathbb{N})\rightarrow \ell^2(\mathbb{N})\oplus \ell^2(\mathbb{N})$, defined by
  	\[
  	T=
  	\begin{bmatrix}
  		U&0\\
  		0&S
  	\end{bmatrix}.
  	\]
  	It is easy to see that $T$ is normal and $\sigma_{\text{ess}}(T)\subseteq C(0,1)$, but neither $T$ nor $T^*$ belongs to $\mathcal{AN}(\ell^2(\mathbb{N})\oplus \ell^2(\mathbb N))$.
  \end{example}
  \begin{theorem}\label{normaloid AN}
      Let $T\in\mathcal{AN}(H)$. If one of the following holds
      \begin{enumerate}
      	\item  $\{\|T\|_{e}-\lambda:\lambda\in\sigma_d(T^*T)\}\in\ell^p(\mathbb{N})\text{ for some }1\leq p<\infty,$
      	\item  $T$ is a normaloid operator,
      \end{enumerate}
     
      then $T$ has a non-trivial hyperinvariant subspace.
  \end{theorem}
  \begin{proof}
      Let $T\in\mathcal{AN}(H).$ By \cite[Corollary 2.11]{VENKU}, we know that $T^*T\in\mathcal{AN}(H)$, and \cite[Theorem 2.4]{RAMpara} implies that $\sigma_{\text{ess}}(T^*T)$ is singleton, say $\{\alpha\}$. Note that $\alpha=\|T\|_{e}$. Now, the result follows from Theorem \ref{Main thm2} and Theorem \ref{Main thm}.
  \end{proof}
  %\begin{theorem}\label{normaloid AN}
  %	Let $T\in\mathcal{AN}(H)$ be a normaloid operator. Then $T$ has a non-trivial hyperinvariant subspace.
  %\end{theorem}
  %\begin{proof}
   %   We know that $T^*T\in\mathcal{AN}(H)$ from \cite[Corollary 2.11]{VENKU}. By \cite[Theorem 2.4]{RAMpara}, we know that $\sigma_{\text{ess}}(T^*T)$ is singelton, say $\{\alpha\}$ and the set $\{\lambda\in\sigma_p(T^*T):\lambda<\alpha\}$ is finite. The result follows from Theorem \ref{Main thm}.
  %\end{proof}
  It is easy to see that the operator in Example \ref{example AN} is an absolutely norm attaining operator. Thus Example \ref{example AN} illustrates that the normaloid assumption in the above theorem is sufficient but not necessary.

  Next, we discuss a particular example of  normaloid $\mathcal{AN}$-operators.
  
  We denote by $L^{\infty}(\mathbb{T})$ the Banach space of all essentially bounded measurable functions on the unit circle $\mathbb{T} =\{z \in \mathbb{C}: |z|=1\}$ and by $L^2(\mathbb{T})$ the Hilbert space of all Lebesgue square integrable functions on $\mathbb{T}$. The Hilbert space containing $L^2(\mathbb{T})$-functions whose negative Fourier coefficients are zero is the Hardy space denoted by $H^2.$
  
  %Let $ L^2(\mathbb{T})$ denote the Hilbert space of all square-integrable functions on the unit circle $\mathbb{T} =\{z \in \mathbb{C}: |z|=1\}$, i.e.,
 % \begin{equation*}
  %	L^2(\mathbb{T}) = \{f : \mathbb{T} \to \mathbb{C}  \ \big\vert \ f \ \text{is measurable} \ \text{and} \int_{\mathbb{T}}^{} |f(z)|^2 d\mu(z) < \infty \},
  %\end{equation*}
  %where $\mu$ is the normalized Lebesgue measure on the Borel sets of $\mathbb{T}$ in the complex plane. It is well known that the functions $e_n = z^n, n = 0, \pm1, \pm2, \dots,$ constitute an orthonormal basis for $L^2(\mathbb{T})$. Let $H^2$ denote all those functions in $L^2$ such that $\langle f, e_n \rangle = 0$ for all $n = -1, -2, \dots$. Then $H^2$ is a closed subspace of $L^2$ and hence a Hilbert space. The Banach space of all essentially bounded measurable functions on $\mathbb{T}$ is denoted by $L^\infty(\mathbb{T})$.
  
  The Laurent operator $L_\varphi$ with the symbol $\varphi\in L^{\infty}(\mathbb{T})$ is the bounded operator on $L^2(\mathbb{T})$ defined by% induced by a function $\varphi \in L^\infty(\mathbb{T})$  given by 
  \[(L_\varphi f)(z) = \varphi(z)f(z),\ z \in \mathbb{T}.\]
  %is called the Laurent operator. 
  Further,  $L_\varphi$ is bounded and $\|L_\varphi \| = \|\varphi\|_\infty$. The Laurent operator induces an  operator called the Toeplitz operator $T_\varphi :H^2 \to H^2$ given by
  \[T_\varphi f = PL_\varphi f, \ \forall f \in H^2,\]
  
  where $P$ is an orthogonal projection of $L^2$ onto $H^2$.
 \begin{corollary}
 	Let  $\varphi \in L^\infty(\mathbb{T})$. If $T_\varphi \in \mathcal{AN}(H^2)$, then $T_\varphi$ has a hyperinvariant subspace.
 \end{corollary}

 \begin{proof}
 	By \cite[Corollary 1, Page 138]{Halmosprobbook}, we know that $r(T_{\phi})=\|T_{\phi}\|=\|\phi\|_{\infty}$, that is, $T_{\phi}$ is normaloid. Hence by Theorem \ref{normaloid AN}, $T_\varphi \in \mathcal{AN}(H^2)$ has a non-trivial hyperinvariant subspace.
 \end{proof}
 The above Corollary is proved in \cite[Corollary 2.2]{GRSSS1}. We also refer to \cite[Theorem 2.4]{GRSSS2} for the information about when is $T_{\varphi}$ an  $\mathcal {AN}$-operator?

Next, we discuss the existence of hyperinvariant subspaces for absolutely minimum attaining operators, a class that lie in the class of operators discussed in the last section. Let us  recall a few necessary details of this class of operators.

 For $T\in \mathcal{B}(H)$, the \textit{minimum modulus} of $T$ is defined by
 $$m(T)=\inf\{\|Tx\|:x\in S_H\}.$$
 We say that $T$ is a \textit{minimum attaining} operator, if there exist $x\in S_H$ such that $m(T)=\|Tx\|$ and  \textit{absolutely minimum attaining} (shortly $\mathcal{AM}$-operator), if for every non-zero closed subspace $M\subseteq H$, $T|_M:M\rightarrow H$ is minimum attaining. We denote the class of all absolutely minimum attaining operators by $\mathcal{AM}(H)$. %This class was defined in \cite{CAR} and characterization of this class is discussed in \cite{GAN1}. 
 For more details about $\mathcal{AM}$-operators, we refer to \cite{RAM,CAR,GAN,GAN1}.
 
 By imitating the same procedure as in Propositions \ref{prop essan}, \ref{prop essnormal} and using results \cite[Theorem 5.8]{GAN1} and \cite[Theorems 3.15, 3.16]{RAM}, we can prove the following results for $\mathcal{AM}$-operators as well.
 \begin{proposition}\label{prop essanAM}
 	Let $T\in\mathcal{B}(H_1,H_2)$. Then any two conditions of the following imply the third condition.
 	\begin{enumerate}
 		\item\label{TAM} $T\in\mathcal{AM}(H_1,H_2)$.
 		\item\label{TadjoAM}  $T^*\in\mathcal{AM}(H_2,H_1).$
 		\item\label{essspeAM} $\sigma_{\text{ess}}(T^*T)=\sigma_{\text{ess}}(TT^*)$.
 	\end{enumerate}
 \end{proposition}
 \begin{proposition}\label{prop essnormalAM}
 	Let $T,T^*\in\mathcal{AM}(H)$. Then $T$ is essentially normal and $\sigma_{\text{ess}}(T)\subseteq C(0,\beta)$, for some $\beta\in\mathbb{R}^+\cup\{0\}$.
 \end{proposition}
 
 Similar to normaloid operators, we propose a new definition that depends on the minimum modulus of an operator.
 \begin{definition}
 	Let $T\in\mathcal{B}(H)$. Then we say $T$ to be minimaloid if $m(T)=d(0,\sigma(T))=\inf\{|\lambda|:\lambda\in\sigma(T)\}$.
 \end{definition}
 Note that every normal operator is minimaloid. The next result relates the class of normaloid operators and minimaloid operators.
 \begin{lemma}\label{minimaloidnormaloidreln}
 	Suppose $T\in\mathcal{B}(H)$ has a bounded inverse. Then $T$ is minimaloid if and only if $T^{-1}$ is normaloid.
 \end{lemma}
 \begin{proof}
 	The proof follows from the fact that $\sigma(T^{-1})=\{\lambda^{-1}:\lambda\in\sigma(T)\}$ and $\|T^{-1}\|=1/m(T)$.
 	$\qedhere$
 \end{proof}
 \begin{theorem}\label{minmaloidAMHIS}
 	Let $T\in\mathcal{AM}(H)$ be minimaloid. Then $T$ has a hyperinvariant subspace.
 \end{theorem}
 \begin{proof} Since $T\in\mathcal{AM}(H)$, we have  $R(T)$ is a closed subspace of $H$ by \cite[Proposition 3.3]{GAN}. We assume that $0\notin\sigma(T)$, otherwise either $N(T)\ne\{0\}$ or $R(T)\ne H$ is a non-trivial hyperinvariant subspace for $T$.
 %	Case (1): Let $0\notin\sigma(T)$. 
 By  \cite[Corollary 4.7]{KUL} and Proposition \ref{minimaloidnormaloidreln}, $T^{-1}\in\mathcal{AN}(H)$ is a normaloid operator.
 	Now the result follows from Theorem \ref{normaloid AN},
 	because if $S\in \mathcal B(H)$ is such that $ST=TS$, then $ST^{-1}=T^{-1}S$.
 	%Case (2): Let $0\in\sigma(T)$. If $N(T)\ne\{0\}$, then $N(T)$ is a hyperinvariant subspace for $T$. On the other hand, if $N(T)=\{0\}$, then $N(T^*)\ne\{0\}$, because $R(T)$ is closed by \cite[Proposition 3.3]{GAN} and if $N(T^*)=\{0\}$ then $T$ must be invertible, which is a contradiction to our assumption. Consequently, $R(T)$ is a non-trivial hyperinvariant subspace for $T.\qedhere$
 \end{proof}
 %\begin{remark}
 %	The operator $T$ defined in Example \ref{necessaryeg} is an $\mathcal{AM}$-operator. Note that as $T$ is not one-to-one, we get $m(T)=0=d(0,\sigma(T)$. Hence $T$ is minimaloid. So from the same example, we can conclude that the condition minimaloid in Thoerem \ref{minmaloidAMHIS} cannot be dropped..
 %\end{remark}

 %\begin{question}
 %	Can we drop the normaloid and minimaloid assumption in Theorem \ref{normaloid AN} and \ref{minmaloidAMHIS}, respectively?
 %\end{question}

   We end this section with a brief discussion about non-trivial hyperinvariant subspace for $\overline{\mathcal{AN}(H)}$ and $\overline{\mathcal{AM}(H)}$, where the closure is taken with respect to the operator norm of $\mathcal{B}(H)$. For a detailed discussion about  $\overline{\mathcal{AN}(H)}$, we refer to \cite{GRSSS3}. Surprisingly, it turns out that $\overline{\mathcal{AN}(H)}=\overline{\mathcal{AM}(H)}$ (see \cite[Theorem 6.10]{GRSSS3} for more details).

  We recall that a partial isometry $V\in \mathcal{AN}(H)$ if and only if either $N(V)$ or $R(V)$ is finite dimensional (\cite[Proposition 3.14]{CAR1}). Furthermore, $V\in \mathcal{AN}(H)$ if and only if $V\in \overline{\mathcal {AN}(H)}$ (\cite[Corollary 3.10]{GRSSS3}) if and only if $V\in \overline{\mathcal{AM}(H)}$ (\cite[Theorem 6.10]{GRSSS3}) if and only $V\in \mathcal{AM}(H)$ (\cite[Proposition 6.9]{GRSSS3}). From these observations we have the following.
  \begin{remark}\label{partialisoANclosure}
      Let $V\in \mathcal B(H)\setminus \mathcal K(H)$ be a partial isometry. Then the following are equivalent.
      \begin{enumerate}
      \item $N(V)$ is finite dimensional
      \item $V\in \mathcal{AN}(H)$
      \item $V\in \overline{\mathcal{AN}(H)}=\overline{\mathcal{AM}(H)}$
      %\item $V\in \overline{\mathcal{AM}(H)}$
      \item $V\in \mathcal{AM}(H)$.      
      \end{enumerate}      
  \end{remark}
 \begin{theorem}
 Let $T\in \mathcal B(H)\setminus \mathcal K(H)$. Then the following are equivalent;
 \begin{enumerate}
 \item\label{singleesspt} $\sigma_{\text{ess}}(T^*T)={\{\alpha}\}$ for some  $\alpha>0$
 \item \label{closureANchar}$T^*T \in \overline{\mathcal{AN}(H)}$
 \item\label{cptpertANpartialiso} there exists a partial isometry $V$ with finite dimensional $N(V)$ and a compact operator $K\in \mathcal K(H)$ such that $T=K+\sqrt{\alpha}V$.
 \end{enumerate}
 \end{theorem}
 \begin{proof}
 The implication (\ref{singleesspt})$\Rightarrow $(\ref{closureANchar}) follows from \cite[Theorem 4.6]{GRSSS3}.  To prove (\ref{closureANchar})$\Rightarrow $(\ref{cptpertANpartialiso}), let us assume that $T^*T \in \overline{\mathcal{AN}(H)}$. Then by  \cite[Lemma 3.14]{GRSSS3}, $|T|\in \overline{\mathcal{AN}(H)}$, and by \cite[Theorem 4.2]{GRSSS3}, there exists positive compact operators $K_1, K_2$ and $\beta \geq 0$ with $K_1K_2=0$ and $K_1\leq \beta I$ such that $|T|=\beta I-K_1+K_2$. If $T=V|T|$ is the polar decomposition of $T$, then we can obtain $T=\sqrt{\alpha}V+K$, where $\beta=\sqrt{\alpha}$ and $K=V(K_2-K_1)$. By \cite[Corollary 3.8]{GRSSS3}, we can conclude that $V\in \overline{\mathcal{AN}(H)}$ and by \cite[Corollary 3.10]{GRSSS3}, we get that $V\in \mathcal{AN}(H)$. By  Remark \ref{partialisoANclosure}, $N(V)$ is finite dimensional.
 
 The implication (\ref{cptpertANpartialiso})$\Rightarrow $(\ref{singleesspt}) can be obtained by computing $T^*T$ and applying the Weyl's theorem on the essential spectrum.
 \end{proof}
 \begin{corollary}
     Let $T=\beta V+K$, where  $\beta \geq 0$, $V\in \mathcal B(H)$ be a partial isometry with finite dimensional null space and $K\in \mathcal K_{p}(H)$ for some $1\leq p<\infty$. Then $T$ has a non-trivial hyperinvariant subspace.
 \end{corollary}
 \begin{proof}
 We have 
 \begin{align*}
 T^*T&=\beta^2V^*V+\beta(V^*K+K^*V)+K^*K\\
     &=\beta^2(I-P_{N(V)})+ \beta(V^*K+K^*V)+K^*K\\
     &=\beta^2I+\Hat{K},
     \end{align*}
     where $\Hat{K}=\beta(V^*K+K^*V)+K^*K-\beta^2P_{N(V)}$. As $N(V)$ is finite dimensional, we get that $P_{N(V)}\in \mathcal F(H)$ and consequently $\Hat{K}\in \mathcal K_{p}(H)$. It is clear that $\sigma_{\text{ess}}(T^*T)={\{\beta^2}\}$ and $T^*T$ satisfies the conditions of Theorem \ref{Main thm2}. Hence $T$ has a non-trivial hyperinvariant subspace.
 \end{proof}

\begin{center}
		Acknowledgements
\end{center}
We are grateful to the referee for the valuable suggestions, which have enhanced the clarity and presentation of this article.

\end{document}